\documentclass [a4 paper,12pt]{article}

\usepackage{bbm}
\usepackage{amsmath,amssymb,amsfonts,amsthm,graphics,graphicx}

\newtheorem{lem}{Lemma}[section]
\newtheorem{thm}[lem]{Theorem}
\newtheorem{cor}[lem]{Corollary}

\newtheorem{con}[lem]{Conjecture}

\newtheorem{defi}[lem]{Definition}

\def\pf{\noindent {\it Proof.} }

\title{Hypoenergetic and strongly hypoenergetic trees
 \footnotetext[1]{Supported by NSFC No.10831001, PCSIRT and the ``973" program.}
 \footnotetext[2]{Supported by NSFC No.10871166, NSFJS and NSFUJS.}}

\author{ \small  Xueliang Li\footnotemark[1], ~~ Hongping Ma\footnotemark[2]\\[5pt]
\small Center for Combinatorics and LPMC-TJKLC,\\
\small Nankai University, Tianjin 300071, P.R. China.\\
\small Email: lxl@nankai.edu.cn; mhp@cfc.nankai.edu.cn}
\date{}

\begin{document}
\maketitle
\begin{abstract}
The energy $E(G)$ of a graph $G$ is defined as the sum of the
absolute values of the eigenvalues of $G$. An $n$-vertex graph is
said to be hypoenergetic if $E(G)<n$ and strongly hypoenergetic if
$E(G)<n-1$. In this paper, we consider hypoenergetic and strongly
hypoenergetic trees. For any given $n$ and $\Delta$, the existence
of both hypoenergetic and strongly hypoenergetic trees of order $n$
and maximum degree $\Delta$ is completely characterized.\\[3mm]
\end{abstract}

\section{Introduction}

Let $G$ be a simple graph with $n$ vertices. Denote by $\Delta$ the
maximum degree of a graph. The eigenvalues $\lambda_{1},
\lambda_{2},\ldots, \lambda_{n}$ of the adjacency matrix $A(G)$ of
$G$ are said to be the eigenvalues of the graph $G$. The nullity of
$G$, denoted by $n_{0}(G)$ (or simply $n_0$), is the multiplicity of
zero in the spectrum of $G$. The $energy$ of $G$ is defined as
$$E=E(G)=\sum_{i=1}^{n}|\lambda_{i}|.$$

For several classes of graphs it has been demonstrated that the
energy exceeds the number of vertices (see, \cite{G1}). In 2007,
Nikiforov \cite{N} showed that for almost all graphs,
$$E=\left(\frac{4}{3\pi}+o(1)\right)n^{3/2}.$$
Thus the number of graphs satisfying the condition $E<n$ is
relatively small. In \cite{GR}, a $hypoenergetic$ $graph$ is defined
to be a (connected) graph $G$ of order $n$ satisfying $E(G)<n$;
whereas in \cite{SRAG} a {\it strongly hypoenergetic graph} is
defined to be a (connected) graph $G$ of order $n$ satisfying
$E(G)<n-1$. For hypoenergetic trees, Gutman et al. \cite{GLSZ}
obtained the following results.

\begin{lem}\cite{GLSZ}\label{lem1.2}
(a) There exist hypoenergetic trees of order $n$ with maximum degree
$\Delta \leq 3$ only for $n=1,3,4,7$ (a single such tree for each
value of $n$, see Figure 1); (b) If $\Delta =4$, then there exist
hypoenergetic trees for all $n\geq 5$, such that $n\equiv k$ $(mod
~4)$, $k=0,1,3$; (c) If $\Delta \geq 5$, then there exist
hypoenergetic trees for all $n\geq \Delta +1$.
\end{lem}

Almost in the same time, Nikiforov in \cite{N2} (see his Theorem 9)
also claimed that he proved the above result (a). However, at the
end of his proof for the case (the maximum eigenvalue)
$\lambda_{1}<\sqrt{7}$, he used the inequality
$\frac{8(n-1)^3}{12n+2}>n^2$, but this is not valid for any $n$.
Therefore, his proof left such a gap, and we hope that he could find
a valid proof.

In \cite{GR} it was reported that the computer search showed that
there exist hypoenergetic trees with $\Delta =4$ and
$n=6,10,14,18,22$, namely for the first five even integers greater
than 2, not divisible by $4$. Based on this observation, Gutman et
al. proposed the following conjecture.

\begin{con}\cite{GLSZ}\label{con1} There exist hypoenergetic trees of
order $n$ with $\Delta=4$ for any $n\equiv 2$ $(mod ~4)$, $n>2$.
Consequently, there exist hypoenergetic trees of order $n$ with
$\Delta =4$ for all $n\geq 5$.
\end{con}

We will give a positive proof to this conjecture later, and
therefore, Lemma \ref{lem1.2} is extended to the following result.

\begin{thm} \label{thm1.0}
(a) There exist hypoenergetic trees of order $n$ with maximum
degree $\Delta \leq 3$ only for $n=1,3,4,7$ (a single such tree
for each value of $n$, see Figure 1); (b) If $\Delta \geq 4$, then
there exist hypoenergetic trees for all $n\geq \Delta +1$.
\end{thm}

To prove Conjecture \ref{con1}, we need the following notations
and preliminary results, which can be found in \cite{SRAG}. Let
$G$ and $H$ be two graphs with disjoint vertex sets, and let $u\in
V(G)$ and $v\in V(H)$. Construct a new graph $G\circ H$ from
copies of $G$ and $H$, by identifying the vertices $u$ and $v$.
Thus $|V(G\circ H)|=|V(G)|+|V(H)|-1$. The graph $G\circ H$ is
known as the $coalescence$ of $G$ and $H$ with respect to $u$ and
$v$.

\begin{thm}\cite{SRAG}\label{thm1.1}
Let $G$, $H$ and $G\circ H$ be graphs as specified above. Then
$E(G\circ H)\leq E(G)+E(H)$. Equality is attained if and only if
either $u$ is an isolated vertex of $G$ or $v$ is an isolated vertex
of $H$ or both.
\end{thm}

Many results on the minimal energy have been obtained for various
classes of graphs. In \cite{HW}, Heuberger and Wagner studied trees
with bounded maximum degree. To state their result, we use the
notion of {\it complete $d$-ary trees}: the complete $d$-ary tree of
height $h-1$ is denoted by $C_h$, i.e., $C_1$ is a single vertex and
$C_h$ has $d$ branches $C_{h-1},\dots,C_{h-1}$. It is convenient to
set $C_0$ to be the empty graph.

\begin{defi}\cite{HW}\label{def1}
$T_{n,d}^{*}$ is the tree with $n$ vertices that can be decomposed
as
\begin{figure}[ht]
\centering
  \setlength{\unitlength}{0.05 mm}%
  \begin{picture}(3060.4, 545.7)(0,0)
  \put(0,0){\includegraphics{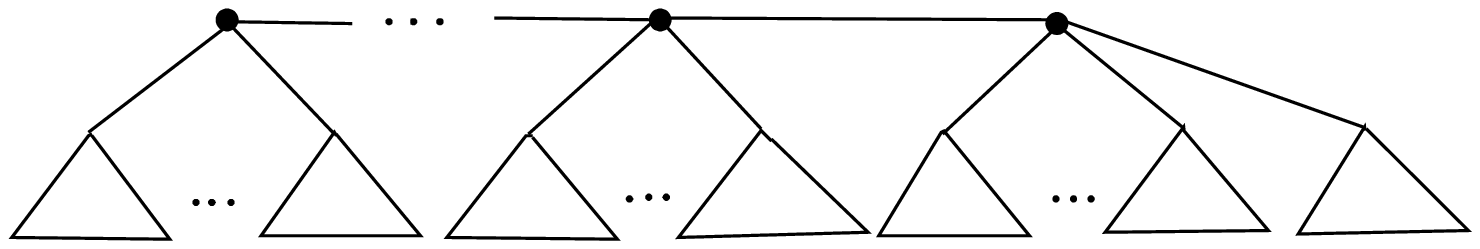}}
  \put(129.16,83.30){\fontsize{8.53}{10.24}\selectfont \makebox(270.0, 60.0)[l]{$B_{0,1}$\strut}}
  \put(619.72,77.69){\fontsize{8.53}{10.24}\selectfont \makebox(330.0, 60.0)[l]{$B_{0,d-1}$\strut}}
  \put(995.79,69.03){\fontsize{8.53}{10.24}\selectfont \makebox(330.0, 60.0)[l]{$B_{l-1,1}$\strut}}
  \put(1466.55,79.73){\fontsize{8.53}{10.24}\selectfont \makebox(390.0, 60.0)[l]{$B_{l-1,d-1}$\strut}}
  \put(1872.25,79.73){\fontsize{8.53}{10.24}\selectfont \makebox(270.0, 60.0)[l]{$B_{l,1}$\strut}}
  \put(2336.75,83.30){\fontsize{8.53}{10.24}\selectfont \makebox(330.0, 60.0)[l]{$B_{l,d-1}$\strut}}
  \put(2750.45,101.13){\fontsize{8.53}{10.24}\selectfont \makebox(270.0, 60.0)[l]{$B_{l,d}$\strut}}
  \end{picture}%
\end{figure}

\noindent with $B_{k,1},\dots,B_{k,d-1}\in \{C_{k},C_{k+2}\}$ for
$0\leq k<l$ and either $B_{l,1}=\dots=B_{l,d}=C_{l-1}$ or
$B_{l,1}=\dots=B_{l,d}=C_{l}$ or $B_{l,1},\dots,B_{l,d}\in
\{C_{l},C_{l+1},C_{l+2}\}$, where at least two of
$B_{l,1},\dots,B_{l,d}$ equal $C_{l+1}$. This representation is
unique, and one has the ``digital expansion"
$$(d-1)n+1=\sum _{k=0}^{l}a_{k}d^{k},$$
where $a_{k}=(d-1)(1+(d+1)r_{k})$ and $0\leq r_{k}\leq d-1$ is the
number of $B_{k,i}$ that are isomorphic to $C_{k+2}$ for $k<l$, and

$\bullet$ $a_l=1$ if $B_{l,1}=\dots=B_{l,d}=C_{l-1}$,

$\bullet$ $a_l=d$ if $B_{l,1}=\dots=B_{l,d}=C_{l}$,

\hangafter=1\hangindent=2.3em $\bullet$ or otherwise
$a_l=d+(d-1)q_l+(d^2-1)r_l$, where $q_l\geq 2$ is the number of
$B_{l,i}$ that  are isomorphic to $C_{l+1}$ and $r_l$ the number of
$B_{l,i}$ that are isomorphic to $C_{l+2}$.
\end{defi}

Let $\mathcal {T}_{n,d}$ be the set of all trees with $n$ vertices
and $\Delta\leq d+1$.

\begin{lem}\cite{HW}\label{lem1.5}
Let $n$ and $d$ be positive integers. Then $T_{n,d}^{*}$ is the
unique (up to isomorphism) tree in $\mathcal {T}_{n,d}$ that
minimizes the energy.
\end{lem}

 We will use Lemma \ref{lem1.5} to obtain strongly
hypoenergetic trees with $\Delta=4$ later.

\section{Main results}

The following result is need in the sequel.

\begin{lem}\cite{G2}\label{lem1} Let $G$ be a graph with
$n$ vertices and $m$ edges. If the nullity of $G$ is $n_{0}$, then
$E(G)\leq \sqrt{2m(n-n_{0})}$.
\end{lem}

From Table 2 of \cite{CDS} we have for the graphs in Figure 1 that
$E(S_1)=0$, $E(S_3)=2.828$, $E(S_4)=3.464$, and $E(W)=6.828$. By
Lemma \ref{lem1.2} we get the following result.

\begin{figure}[ht]
\centering
  \setlength{\unitlength}{0.05 mm}%
  \begin{picture}(1905.7, 379.4)(0,0)
  \put(0,0){\includegraphics{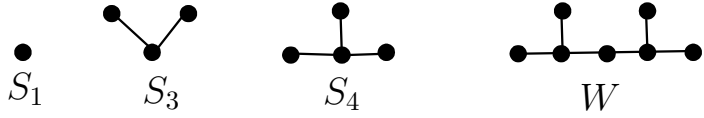}}
  \put(40.00,65.32){\fontsize{14.23}{17.07}\selectfont \makebox(250.0, 100.0)[l]{$S_1$\strut}}
  \put(397.59,54.70){\fontsize{14.23}{17.07}\selectfont \makebox(250.0, 100.0)[l]{$S_3$\strut}}
  \put(872.02,51.16){\fontsize{14.23}{17.07}\selectfont \makebox(250.0, 100.0)[l]{$S_4$\strut}}
  \put(1548.27,37.00){\fontsize{14.23}{17.07}\selectfont \makebox(150.0, 100.0)[l]{$W$\strut}}
  \end{picture}%
\caption{The hypoenergetic trees with maximum degree at most $3$.}
\end{figure}

\begin{thm}\label{thm1}
There do not exist any strongly hypoenergetic trees with maximum
degree at most $3$.
\end{thm}

For trees with maximum degree at least $4$, we have the following
results.
\begin{lem}\label{thm2}
(1) If $\Delta =4$, then there exist $n$-vertex strongly
hypoenergetic trees for all $n>5$ such that $n\equiv 1$ $(mod
~4)$; (2) If $\Delta =5$, then there exist $n$-vertex strongly
hypoenergetic trees for $n=6$ and all $n\geq 9$, but there do not
exist any strongly hypoenergetic trees for $n=7$ and $8$; (3) If
$\Delta \geq 6$, then there exist $n$-vertex strongly
hypoenergetic trees for all $n\geq \Delta +1$.
\end{lem}
\begin{proof} Let $T$ be a tree of order $n$. By Lemma \ref{lem1},
\begin{align}E(T)\leq \sqrt{2(n-1)(n-n_{0})}. \label{1}\end{align}
Equality in \eqref{1} is attained if and only if $T$ is the
$n$-vertex star $S_n$. Note that $E(S_n)=2\sqrt{n-1}<n-1$ for
$n>5$ and $E(S_5)=4=n-1$, i.e., $S_n$ is strongly hypoenergetic
for $n>5$. Therefore, in what follows, without loss of generality
we may assume that $T$ is not a star, which implies that the
inequality in \eqref{1} is strict. Now, if
$\sqrt{2(n-1)(n-n_{0})}\leq n-1$, or equivalently,
\begin{align} n_0\geq\frac{n+1}{2},\label{2}\end{align} then the tree
$T$ will necessarily be strongly hypoenergetic. Fiorini et al.
\cite{FGS} proved that the maximum nullity of a tree with given
values of $n$ and $\Delta$ is
\begin{align}n-2\left\lceil\frac{n-1}{\Delta}\right\rceil, \label{3}\end{align}
and showed how trees with such nullity can be constructed.

Combining \eqref{2} and \eqref{3} we arrive at the condition
$n-2\left\lceil\frac{n-1}{\Delta}\right\rceil\geq \frac{n+1}{2}$, or
equivalently,
\begin{align}\left\lceil\frac{n-1}{\Delta}\right\rceil \leq \frac{n-1}{4},\label{4}\end{align}
which, if satisfied, implies the existence of at least one
strongly hypoenergetic tree with $n$ vertices and maximum degree
$\Delta$.

Observe that
\begin{equation*}\left\lceil\frac{n-1}{\Delta}\right\rceil=
\begin{cases}
n/ \Delta,         & \text{if $n\equiv 0 ~(mod ~\Delta)$},\\
(n-1)/ \Delta,     & \text{if $n\equiv 1 ~(mod ~\Delta)$},\\
 (n-k)/ \Delta +1, & \text{if $n\equiv k ~(mod ~\Delta),
~ k=2,3,\ldots,\Delta-1.$}
\end{cases}
\end{equation*}

Hence in the case $n\equiv 1 ~(mod ~\Delta)$, the inequality
\eqref{4} holds for all $\Delta \geq 4$.

If $n\equiv 0 ~(mod ~\Delta)$, then the inequality \eqref{4} is
transformed into $(\Delta -4)n-\Delta \geq 0$, which is always valid
for all $n\geq \Delta \geq 5$.

Now we consider the case \text{$n\equiv k ~(mod ~\Delta), ~
k=2,3,\ldots,\Delta-1$}. Since $T$ is a tree, we only need to
consider $n\geq \Delta +k$. Then the inequality \eqref{4} is
transformed into
\begin{align}\frac{n-k}{\Delta} \leq
\frac{n-5}{4}.\label{5}\end{align}

If $\Delta \geq 7$, then $\frac{n-k}{\Delta}\leq \frac{n-2}{7}$, and
it is easy to check that the inequality $\frac{n-2}{7}\leq
\frac{n-5}{4}$ holds for all $n\geq 9$.

If $\Delta = 6$, then $\frac{n-k}{\Delta}\leq \frac{n-2}{6}$, and
it is easy to check that the inequality $\frac{n-2}{6}\leq
\frac{n-5}{4}$ holds for all $n\geq 11$. For $n=9$ or $10$, we
have $n=\Delta +k$, so the inequality \eqref{5} also holds.
Although the inequality \eqref{5} does not hold for $n=8$, we know
that there exists a unique tree of order $8$ with $\Delta=6$, and
the energy of the tree is $6.774$ (see Table 2 in \cite{CDS}),
which is less than $n-1=7$.

Now suppose $\Delta=5$. If $k=4$, then
$\frac{n-k}{\Delta}=\frac{n-4}{5}$, and it is easy to check that the
inequality $\frac{n-4}{5}\leq \frac{n-5}{4}$ holds for all $n\geq
9$. If $k=3$, then $\frac{n-k}{\Delta}=\frac{n-3}{5}$, and it is
easy to check that the inequality $\frac{n-3}{5}\leq \frac{n-5}{4}$
holds for all $n\geq 13$. By \cite{CDS} (Table 2), there are $3$
trees with energy $7.114, 7.212$ and $8.152$, respectively, of order
$8$ with $\Delta=5$. Hence there do not exist any strongly
hypoenergetic trees of order $n=8$ with $\Delta=5$. If $k=2$, then
$\frac{n-k}{\Delta}=\frac{n-2}{5}$, and it is easy to check that the
inequality $\frac{n-2}{5}\leq \frac{n-5}{4}$ holds for all $n\geq
17$. By \cite{CDS} (Table 2), there exists a unique tree of order
$7$ with $\Delta=5$, and the energy of the tree is $6.324$, which is
to say that the tree is not strongly hypoenergetic. Finally, we
construct a strongly hypoenergetic tree of order $12$ with $\Delta
=5$. As stated above, there exists a tree, denoted by $T_{11}$, of
order $11$ with $\Delta =5$ and its nullity
$n_0=n-2\left\lceil\frac{n-1}{\Delta}\right\rceil=11-2\left\lceil\frac{10}{5}\right\rceil=7$.
Then by the inequality \eqref{1}, $E(T_{11})\leq
\sqrt{2(n-1)(n-n_{0})}=\sqrt{2(11-1)(11-7)}<9$. Let $T_2$ be the
tree of order $2$, $v\in V(T_2)$ and $u$ a leaf vertex in $T_{11}$.
Clearly, $E(T_2)=2$. Let $T_{11}\circ T_2$ be the coalescence of
$T_{11}$ and $T_2$ with respect to $u$ and $v$. Thus by Theorem
\ref{thm1.1}, $E(T_{11}\circ T_2)\leq E(T_{11})+E(T_2)<9+2=11$.
Obviously, $T_{11}\circ T_2$ is a tree of order $12$ with maximum
degree $5$ and so it is a desired tree. The proof is now complete.
\end{proof}

\noindent{\bf Proof of Conjecture \ref{con1}.} Suppose $n\equiv 2$
$(mod ~4)$, $n>2$. If $n=6$, then form Table 2 of \cite{CDS} , there
exists a unique tree $T_6$ of order $6$ with $\Delta =4$, and
$E(T_6)=5.818<6$, i.e., $T_6$ is hypoenergetic. Note that $n\equiv
2$ $(mod ~4)$ with $n>6$ implies that $n-5\equiv 1$ $(mod ~4)$. If
$n>10$, by Theorem \ref{thm2} (1), there exists a strongly
hypoenergetic tree, denoted by $T_{n-5}$, of order $n-5>5$ with
$\Delta =4$. Let $T_5$ be the 5-vertex star. Then the maximum degree
of $T_5$ is $4$ and $E(T_5)=4=n-1$. Hence $E(T_{n-5})\leq n-6$ for
all $n\geq 10$. Let $u$ be a leaf vertex in $T_6$ and $v$ a leaf
vertex in $T_{n-5}$ with $n\geq 10$. Then, by Theorem \ref{thm1.1},
for the coalescence $T_6\circ T_{n-5}$ of $T_6$ and $T_{n-5}$ with
respect to $u$ and $v$, we have $E(T_6\circ T_{n-5})\leq
E(T_6)+E(T_{n-5})<6+(n-6)=n$. Obviously, $T_6\circ T_{n-5}$ is a
tree of order $n$ with $\Delta=4$ and so it is a desired tree. The
proof is thus complete. ~~~~~~~~~~~~~~~~ $\Box$

\begin{figure}[ht]
\centering
 \setlength{\unitlength}{0.05 mm}%
  \begin{picture}(6820.5, 2191.6)(0,0)
  \put(0,0){\includegraphics{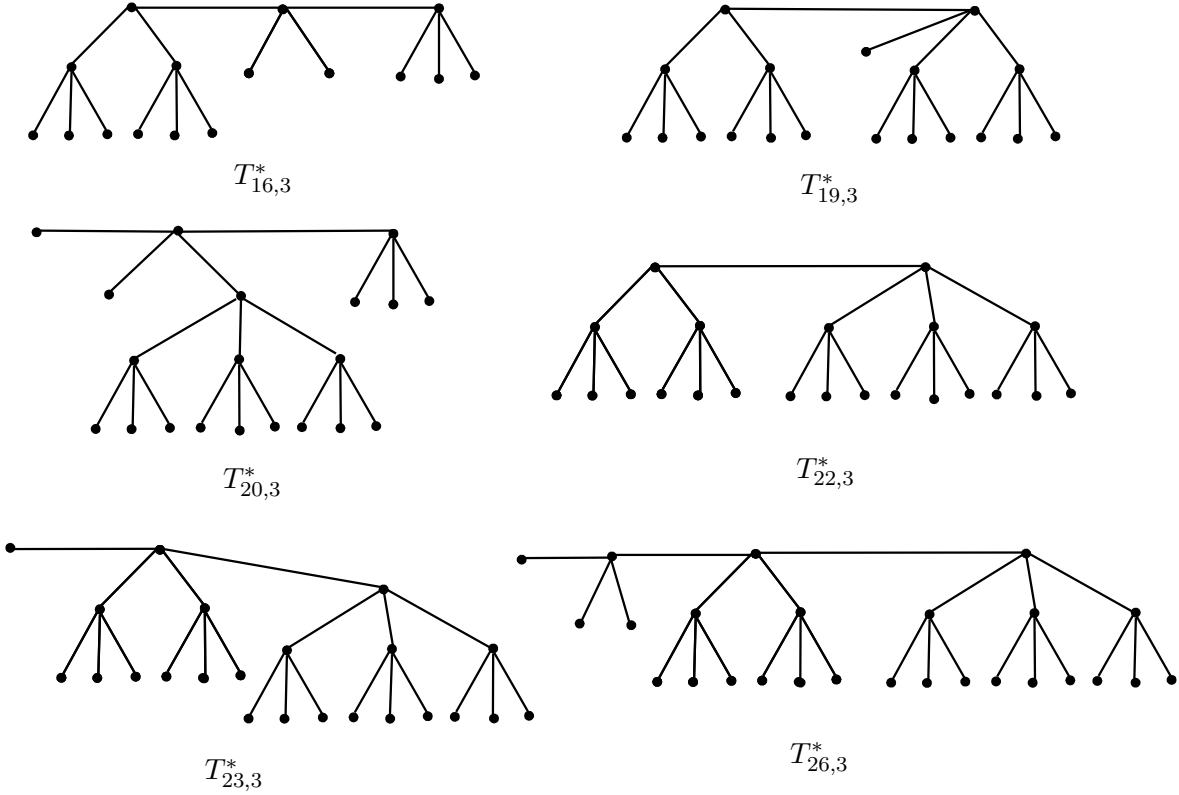}}
  \put(611.29,1615.16){\fontsize{11.38}{13.66}\selectfont \makebox(640.0, 80.0)[l]{ $T_{16,3}^{*}$ \strut}}
  \put(2102.84,1589.03){\fontsize{11.38}{13.66}\selectfont \makebox(640.0, 80.0)[l]{ $T_{19,3}^{*}$ \strut}}
  \put(584.07,810.59){\fontsize{11.38}{13.66}\selectfont \makebox(640.0, 80.0)[l]{ $T_{20,3}^{*}$ \strut}}
  \put(2091.95,837.81){\fontsize{11.38}{13.66}\selectfont \makebox(640.0, 80.0)[l]{ $T_{22,3}^{*}$ \strut}}
  \put(535.08,37.60){\fontsize{11.38}{13.66}\selectfont \makebox(640.0, 80.0)[l]{ $T_{23,3}^{*}$ \strut}}
  \put(2075.62,81.15){\fontsize{11.38}{13.66}\selectfont \makebox(640.0, 80.0)[l]{ $T_{26,3}^{*}$ \strut}}
  \end{picture}%
\caption{Trees $T_{n,3}^{*}$ for $n=16,19,20,22,23$ and $26$.}
\end{figure}

\begin{center}
\begin{tabular}{l l |l l| l l}
\multicolumn{2}{l}{\text{Table 1}}\\[5pt]\hline
$n$  &  $E(T_{n,3}^{*})$  & $n$   &  $E(T_{n,3}^{*})$ & $n$ &
$E(T_{n,3}^{*})$\\\hline
10  &  9.61686 & 11  & 10.36308 & 12  & 11.13490\\
14  & 13.39786 & 15  & 14.26512 & 16  & 15.01712\\
18  & 17.24606 & 19  & 18.13157 & 20  & 18.86727\\
22  & 21.06862 & 23  & 21.96975 &  & \\
26  & 24.87008 & & & \\\hline
\end{tabular}
\end{center}

In the following, we consider strongly hypoenergetic trees for the
remaining case $\Delta =4$ and $n\equiv k$ $(mod ~4)$, $k=0,2,3$. By
\cite{CDS} (Table 2), there do not exist any strongly hypoenergetic
trees with $\Delta =4$ for $n=6,7,8$. By Lemma \ref{lem1.5},
$T_{n,3}^{*}$ is the unique (up to isomorphism) tree in $\mathcal
{T}_{n,3}$ that minimizes the energy. Obviously, $\Delta
(T_{n,3}^{*})=4$. Trees $T_{16,3}^{*}$, $T_{19,3}^{*}$,
$T_{20,3}^{*}$, $T_{22,3}^{*}$, $T_{23,3}^{*}$ and $T_{26,3}^{*}$
are showed in Figure 2. From Table 1, we know that there do not
exist any strongly hypoenergetic trees with $\Delta =4$ for
$n=10,11,12,14,15,16,18,19,22$ and $T_{20,3}^{*}$, $T_{23,3}^{*}$
and $T_{26,3}^{*}$ are strongly hypoenergetic. Then the following
result can be deduced.
\begin{lem}\label{lem}
If $\Delta =4$, then there exist $n$-vertex strongly hypoenergetic
trees for all $n$ such that $n\equiv 0$ $(mod ~4)$ and $n\geq 20$ or
$n\equiv 2$ $(mod ~4)$ and $n\geq 26$ or $n\equiv 3$ $(mod ~4)$ and
$n\geq 23$.
\end{lem}
\pf There are two ways to prove this result: one way is similar to
the proof of Conjecture \ref{con1}. Another simple way is as
follows: Starting from the strongly hypoenergetic trees
$T^*=T_{20,3}^{*}$, $T_{23,3}^{*}$ and $T_{26,3}^{*}$, respectively,
we do the coalescence operation of $T^*$ and the 5-vertex star $S_5$
at a leaf of each of the two trees. It is easy to check that
$T^*\circ S_5$ is strongly hypoenergetic, which follows from Theorem
\ref{thm1.1} and the fact that $E(S_5)=4$ and $T^*$ is strongly
hypoenergetic. By consecutively doing the coalescence operation
$(\cdots ((T^* \circ S_5)\circ S_5) \cdots) \circ S_5$, we can
finish the proof.\qed

Note that both Conjecture \ref{con1} and Lemma \ref{lem} can be
proved in these two ways, however, they may result in different
hypoenergetic trees and strongly hypoenergetic trees. So, both ways
are useful for producing more such trees.

From the above discussion for small $n$ and Lemma \ref{lem} for
large $n$, the result (1) in Lemma \ref{thm2} is now extended as
follows.
\begin{cor} \label{cor}
Suppose $\Delta =4$ and $n\geq 5$. Then there exist $n$-vertex
strongly hypoenergetic trees only for $n=9,13,17,20,21$ and $n\geq
23$.
\end{cor}

Combining Lemma \ref{thm2} and Corollary \ref{cor} we finally arrive
at:
\begin{thm}\label{thm3}
(1) If $\Delta =4$ and $n\geq 5$. Then there exist $n$-vertex
strongly hypoenergetic trees only for $n=9,13,17,20,21$ and $n\geq
23$; (2) If $\Delta =5$ and $n\geq 6$, then there exist $n$-vertex
strongly hypoenergetic trees only for $n=6$ and $n\geq 9$; (3) If
$\Delta \geq 6$, then there exist $n$-vertex strongly hypoenergetic
trees for all $n\geq \Delta +1$.
\end{thm}

\end{document}